\documentclass[11pt]{amsart}
\usepackage{amsmath}
\usepackage{amsfonts}
\DeclareMathOperator{\arcsinh}{arcsinh}
\usepackage{amscd}
\usepackage{amsthm}
\usepackage{amssymb} \usepackage{latexsym}
\usepackage{eufrak}
\usepackage{euscript}
\usepackage{epsfig}
\usepackage{graphics}
\usepackage{array}
\usepackage{enumerate}
\usepackage{dsfont}
\usepackage{color}
\usepackage{wasysym}
\usepackage{hyperref}
\usepackage{pdfsync}

\newcommand{\bel}[1]{\begin{equation}\label{#1}}

\newcommand{\be}{\begin{equation}}

\newcommand{\ba}{\begin{eqnarray}}
\newcommand{\ea}{\end{eqnarray}}

\newcommand{\qe}{\end{equation}}
\newcommand{\R}{{\mathbb R}}
\newcommand{\N}{{\mathbb N}}
\newcommand{\Z}{{\mathbb Z}}

\newcommand{\osc}{{\mathrm{osc}}}
\newcommand{\CDE}{{\mathrm CDE}}

\newcommand{\Hmm}[1]{\leavevmode{\marginpar{\tiny%
$\hbox to 0mm{\hspace*{-0.5mm}$\leftarrow$\hss}%
\vcenter{\vrule depth 0.1mm height 0.1mm width \the\marginparwidth}%
\hbox to
0mm{\hss$\rightarrow$\hspace*{-0.5mm}}$\\\relax\raggedright #1}}}

\theoremstyle{theorem}
\newtheorem{thm}{Theorem}[section]

\theoremstyle{example}
\newtheorem{example}{Example}[section]
\theoremstyle{corollary}
\newtheorem{coro}{Corollary}[section]
\theoremstyle{lemma}
\newtheorem{lemma}{Lemma}[section]
\theoremstyle{definition}
\newtheorem{defi}{Definition}[section]
\theoremstyle{proof}

\theoremstyle{remark}
\newtheorem{rem}{Remark}[section]

\begin{document}

\title{Davies-Gaffney-Grigor'yan Lemma on Graphs}
\author{Frank Bauer}
\email{fbauer@math.harvard.edu}
\address{Harvard University Department of Mathematics\\
One Oxford Street, Cambridge MA 02138}
\author{Bobo Hua}
\email{bobohua@fudan.edu.cn}
\address{School of Mathematical Sciences, LMNS, Fudan University, Shanghai 200433, China; Shanghai Center for Mathematical Sciences, Fudan University, Shanghai 200433, China;
Max Planck Institute for Mathematics in the Sciences, 04103 Leipzig,
Germany.}
\author{Shing-Tung Yau}
\email{yau@math.harvard.edu}
\address{Harvard University Department of Mathematics\\
One Oxford Street, Cambridge MA 02138}

\thanks{FB was partially supported by the Alexander von Humboldt foundation and partially supported by the NSF Grant DMS-0804454 Differential Equations in Geometry. FB and STY acknowledge support by the University of Pennsylvania/Air Force Office of Scientific Research grant "Geometry and Topology of Complex Networks", 
Award \#561009/FA9550-13-1-0097. BH was supported by NSFC, grant no. 11401106.}
\begin{abstract}We prove a variant of the Davies-Gaffney-Grigor'yan Lemma for the continuous time heat kernel on graphs. We use it together with the Li-Yau inequality, to obtain strong heat kernel estimates for graphs satisfying the exponential curvature dimension inequality.
\end{abstract}
\maketitle

\section{Introduction and main results}
\subsection{Introduction}
The Davies-Gaffney-Grigor'yan Lemma (DGG Lemma for short) on manifolds can be stated in the form

\begin{lemma}[Davies-Gaffney-Grigor'yan]\label{DGGManifold}
Let $M$ be a complete Riemannian manifold and $p_t(x,y)$ the minimal heat kernel on $M$. For any two measurable  subsets $B_1$ and $B_2$ of $M$ and $t>0,$ we have
\begin{equation}\label{e:DGG Riemannian}\int_{B_1}\int_{B_2} p_t(x,y)d\mathrm{vol}(x)d\mathrm{vol}(y) \leq \sqrt{\mathrm{vol}(B_1)\mathrm{vol}(B_2)}\exp\left(-\mu t\right)\exp\left(- \frac{d^2(B_1,B_2)}{4t}\right),\end{equation}
where $\mu$ is the greatest lower bound of the $L^2$-spectrum of the Laplacian on $M$ and $d(B_1,B_2)=\inf_{x_1\in B_1,x_2\in B_2}d(x_1,x_2)$ the distance between $B_1$ and $B_2$.
\end{lemma}
A lemma of this type appeared for the first time in a paper of Davies \cite{Davies92}, see also Li and Yau's paper \cite{LiYau86} for an earlier version of this lemma. However Davies mentions that the idea goes back to Gaffney \cite{Gaffney59}. Later the lemma was improved by Grigor'yan \cite{Grigoryan94} who introduced the term $\exp(-\mu t)$ on the right hand side. If $\mu>0$ (for instance for Hyperbolic spaces with constant negative sectional curvature) the term $\exp(-\mu t)$ is particularly important since it gives asymptotically the correct speed of decay of the heat kernel.

The DGG Lemma on Riemannian manifolds is of fundamental importance. Because of its generality (note that no assumptions on the geometry of the manifold are made in Lemma \ref{DGGManifold}), it can be applied in many different situations. Among other applications it was used to obtain eigenvalue estimates \cite{Chung-Grigoryan-Yau96} and in combination with the Li-Yau inequality it yields strong heat kernel estimates \cite{LiYau86, Li12}.

In view of its importance, the question is whether one can prove the DGG Lemma for graphs. The answer to this question is negative. Indeed, it was shown by Coulhon and Sikora \cite{Coulhon08} in a very general setting that for nonnegative self-adjoint operators on general metric measure spaces the DGG Lemma is equivalent to the finite propagation speed property of the wave equation. In particular, the results in \cite{Coulhon08} can be applied in the graph setting. However, it is well-known that for graphs the wave equation does not have the finite propagation speed property, see Friedman-Tillich \cite[pp.249]{FriedmanTillich04}.

The main contribution of this paper is that, despite this negative answer, we are surprisingly  able to prove a variant of the DGG Lemma for the continuous time heat kernel on graphs that approximates the DGG Lemma on manifolds if the time $t$ is big compared to the distance $d$. Moreover we demonstrate the power of the DGG Lemma by obtaining novel heat kernel and eigenvalue estimates.
\subsection{Main results and organization of the paper}
In the following we state and discuss the main results of our paper in detail. For the precise definitions of the quantities used we refer to Section \ref{Notation} and Section \ref{Section Li-Yau}. Our main result is:

\begin{thm}[Davies-Gaffney-Grigor'yan Lemma on graphs]\label{thm:Davies}
Let $G$ be an infinite graph equipped with a measure $m$ and $p_t(x,y)$ be the minimal heat
kernel of $G.$ For any $0<\gamma< 1$ there exists a constant  $\alpha(\gamma)\geq 1$ such that for any subsets $B_1,B_2\subset G$ and $t\geq 0,$
\begin{eqnarray}\label{e:Davies estimate}
\sum_{x\in B_1}\sum_{y\in B_2}p_t(x,y)m(x)m(y)\nonumber&\leq &
\sqrt{m(B_1)m(B_2)}e^{-(1-\gamma)\mu t} \\&&\times \exp\left(-\zeta(\alpha D_m t+1,d(B_1,B_2))\right),
\end{eqnarray} where $\mu$ is the greatest lower bound of the $\ell^2$-spectrum of the graph Laplacian, $d(B_1,B_2)$ is the distance between $B_1$ and $B_2$ and $\zeta(t,d)=d\arcsinh\left(\frac{d}{t}\right)-\sqrt{d^2+t^2}+t$.
Moreover, for the case $\gamma=1,$ we have
\begin{equation}\label{e: Davies gamma=1}
\sum_{x\in B_1}\sum_{y\in B_2}p_t(x,y)m(x)m(y)\leq
\sqrt{m(B_1)m(B_2)}\exp\left(-\frac12\zeta(D_mt,d(B_1,B_2))\right).
\end{equation}
\end{thm}

\begin{rem}
\begin{enumerate}[(a)]
\item The function $\zeta$ in Theorem \ref{thm:Davies} is defined as a Legendre associate and appears naturally in the graph setting, see for example \cite{Davies93, Pang93, Delmotte99}. In view of Lemma \ref{DGGManifold}, $\zeta$ should be comparable to $d^2/2t$. It is not difficult to see that for small $t/d$  the estimates, \eqref{e:Davies estimate} and \eqref{e: Davies gamma=1}, are not true if one replaces $\zeta$ by $\frac{d^2}{2t}$, see \cite{Pang93}. However, for large $t/d$ one can show that $\zeta$ behaves like $\frac{d^2}{2t}$, see \eqref{e:zeta estimate}.

\item
Compared to the Riemannian case, Lemma \ref{DGGManifold}, it would be desirable to prove Theorem \ref{thm:Davies} for  $\gamma=0.$ However on graphs we cannot obtain this result since, for $\gamma=0$, our strategy to find a nontrivial solution for \eqref{e:K(t,x)} in the integral maximum principle (Lemma \ref{l:monotonicity Delmotte}) breaks down. Nevertheless, we can recover the part of the exponential factor (up to a parameter $\gamma\in(0,1]$) which is nontrivial in applications.

\item In the special case $\gamma=1$, the theorem can be directly derived from the results in Delmotte \cite{Delmotte99}. However, it is important to obtain the the exponential factor in $\mu$ on the right hand side.  In this  case, i.e. $0<\gamma<1,$ one cannot use the results of Delmotte and a more delicate argument is needed. We prove a new variant of integral maximum principle on graphs, Lemma \ref{l:monotonicity Delmotte}, that involves the exponential factor in $\mu$. Moreover, we construct nontrivial solutions which satisfy the condition \eqref{e:K(t,x)} in the new integral maximum principle, see Lemma \ref{l:1-Lipschitz}. This is non-trivial for $0<\gamma<1,$ and we need to rescale and shift the time and make use of the crucial fact that on graphs the combinatorial distance function can only attain integer values.
\item Discrete time versions of the integral maximum principle and the DGG Lemma for $\gamma = 1$ were proved in \cite{Coulhon05}.
\end{enumerate}
\end{rem}

In \cite{LiYau86}, Li and Yau obtained their famous heat kernel estimates for mainfolds with Ricci curvature bounded from below by $-K$ for some $K\geq 0$. It was open for a long time whether similar heat kernel estimates hold on graphs. One particular problem was that, on graphs, it is not apparent which the right notion of Ricci curvature is. Here we solve this open problem and prove Li and Yau's heat kernel estimates for graphs satisfying the exponential curvature dimension inequality on graphs which was introduced in \cite{Bauer13}. In the proof of the heat kernel estimate we combine the Harnack inequality, which follows from the Li-Yau inequality, with the DGG Lemma (Theorem \ref{thm:Davies}).

\begin{thm}\label{thm:heat kernel estimate} Let $\epsilon>0, 0<\gamma\leq 1, \beta>0$ and $p_t(x,y)$ be the minimal heat kernel of $G$.  If $G$ satisfies the curvature dimension inequality $\CDE (n,-K),$ then there exist constants $C_1(n,\epsilon,\beta, \gamma, D_m, m_\mathrm{max},\mu_{\mathrm{min}})$, $C_2(n,\epsilon,\beta, \gamma, D_m, m_\mathrm{max},\mu_{\mathrm{min}})$ and $C_3(\gamma,\beta,D_m)$ such that
$$p_t(x,y)\leq C_1\frac{\exp(-(1-\gamma)\mu t)}{\sqrt{m (B_x(\sqrt t))m (B_y(\sqrt t))}}\exp\left(-\frac{C_3d^2(x,y)}{4(1+2\epsilon)t}+C_2\sqrt{Knt}\right),$$ for any $x,y\in G$ and $t\geq \beta d(x,y)\vee1$.
\end{thm}
In Theorem \ref{thm:heat kernel estimate}, we only assume the exponential curvature dimension inequality. Delmotte \cite{Delmotte99} proved the special case ($K=0$ and $\gamma=1$) of the heat kernel estimate in Theorem \ref{thm:heat kernel estimate} by assuming the volume doubling property and the Poincar\'e inequality. In contrast to the volume doubling property and the Poincar\'e inequality, the exponential curvature dimension inequality $CDE(n,-K)$ is a local condition. The advantage is that the exponential curvature dimension inequality can more easily be verified at the cost of being less robust to local perturbations. On Riemannian manifolds it is well known that nonnegative Ricci curvature implies the volume doubling property and the Poincar\'e inequality. However on graphs it is still an open problem weather $CDE(n,0)$ implies these properties.

The paper is organized as follows. In Section \ref{Section Li-Yau} we review the Li-Yau inequality on graphs introduced in \cite{Bauer13} and derive some interesting corollaries of it. In particular, we prove Yau's Liouville theorem, Cheng's Liouville theorem and Cheng's eigenvalue estimate on graphs. While Yau's Liouville theorem was already known under slightly different assumptions, Cheng's Liouville theorem seems to be only known in very special cases (for instance for lattices or Cayley graphs \cite{Hebisch93}). In Section \ref{Section DGG} we prove the DGG Lemma by establishing our main tool the integral maximum principle. In Section \ref{Section Applications DGG} we use the DGG Lemma to prove the heat kernel estimates of the Li-Yau type and as a corollary we derive new heat kernel estimates for finite graphs. Moreover we show how the DGG Lemma can be used to give a purely discrete proof of higher order eigenvalue estimates in terms of the distances between subsets of a finite graph.

\subsection{Setting}\label{Notation}
In this subsection we introduce the setting used throughout this paper. Let $G=G(V,E)$ be a locally finite, connected graph with vertex set $V$ and edge set $E$.  We consider a symmetric weight function $\mu:V\times V\to [0,\infty)$ that satisfies $\mu_{xy}>0$ if and only if $x$ and $y$ are neighbors, in symbols $x\sim y$. Moreover we assume that this weight function satisfies $$\mu_{\mathrm{min}}:= \inf_{(x,y)\in E}\mu_{xy} >0$$ and
$$\deg(x):=\sum_{y\in V}\mu_{xy}<\infty$$ for all $x\in V$.

Let $m: V \to \mathbb{R}_+$ be an arbitrary measure on the vertex set $V$ and let $m_{\mathrm{max}}:= \sup_{x\in V}m(x)$ and  $m_{\mathrm{min}}:= \inf_{x\in V}m(x)$ .  We denote by $C(V)$ the space of real functions on $V,$  by $\ell^p(V,m)=\{f\in C(V): \sum_{x\in V}|f(x)|^pm(x)<\infty\},$ $1\leq p<\infty$, the space of $\ell^p$ integrable functions on $V$ with respect to the measure $m$ (For $p=\infty,$ $\ell^{\infty}(V,m)=\{f\in C(V):\sup_{x\in V}|f(x)|<\infty\}$). For the Hilbert space $\ell^2(V,m),$ we write the inner product as $(f,g)_{\ell^2(V,m)}=\sum_{x\in V} f(x)g(x)m(x).$ We define the Laplace operator $\Delta: C(V) \to C(V)$ with respect to $m$ pointwise by
\[
\Delta f (x) = \frac{1}{m(x)} \sum_{y\in V} \mu_{xy}(f(y) - f(x)) \quad \forall\ x\in V,
\]which coincides with the generator of the Dirichlet form $$f\mapsto \frac12\sum_{x,y\in V}\mu_{xy}|f(x)-f(y)|^2$$ with respect to $\ell^2(V,m)$ on its domain, see Keller-Lenz \cite{KellerLenz12}.
The two most natural choices are $m(x) =\deg(x)$ for all $x\in V$ and $m \equiv 1$. In the first case we obtain the normalized Laplace operator and in the second case the combinatorial Laplace operator, respectively.
It will be useful to define:
\begin{equation}\label{a:on Dm}
D_\mu = \max_{x,y \in V, (x,y)\in E}\frac{\deg(x)}{\mu_{xy}} \text{ and }D_m = \max_{x \in V} \frac{\deg(x)}{m(x)}.
\end{equation}

\section{The Li-Yau inequality on graphs}\label{Section Li-Yau}
In 1975, Yau \cite{Yau75} proved a Liouville type theorem for positive harmonic functions on Riemannian manifolds with nonnegative Ricci curvature and together with Cheng, Yau \cite{ChengYau75} used Bochner's technique to derive the gradient estimate for positive harmonic functions on such manifolds, which yields Cheng's Liouville theorem on sublinear growth harmonic functions. Later on, Li and Yau \cite{LiYau86} derived the parabolic gradient estimate for positive solutions to the heat equations, the so-called Li-Yau inequality. In general, gradient estimates proved to be one of the most powerful tools in geometric analysis. For instance they played a key role in the proof of the Poincar\'e conjecture.

It was open for a long time to prove a Li-Yau inequality on graphs. The two main obstacles were that firstly the chain rule is not available on graphs and secondly it is non-trivial to find the right notion of curvature in the discrete setting. Recently progress was made and a Li-Yau inequality and the corresponding Harnack inequality on graphs were obtained in \cite{Bauer13} by introducing the so-called exponential curvature dimension inequality.
\subsection{The exponential curvature dimension inequality}
Following the work of Bakry and Emery \cite{Bakry85}, there are two natural bilinear forms associated to the Laplacian.
\begin{defi}
The gradient form $\Gamma$ is defined by
\begin{align*}
\Gamma(f,g)(x) & = \frac{1}{2}\big(\Delta(fg) - f \Delta(g) - \Delta(f)g\big)(x) \\ &= \frac{1}{2m(x)} \sum_{y\in V} \mu_{xy} (f(y) - f(x))(g(y)-g(x)).
\end{align*}

The iterated gradient form is defined by
\[
\Gamma_2(f,g)(x) = \frac{1}{2}(\Delta \Gamma(f,g)  - \Gamma(f, \Delta g) - \Gamma(\Delta f, g))(x),
\]
For simplicity, we write $\Gamma(f) = \Gamma(f,f)$ and $\Gamma_2(f) = \Gamma_2(f,f)$.
\end{defi}
Using these bilinear forms one can define the curvature dimension inequality.
\begin{defi}
A graph $G$ satisfies the curvature dimension inequality $CD(n,K)$ if, for any function $f$

\[
\Gamma_2 (f) \geq \frac{1}{n} (\Delta f)^2 + K\Gamma(f). \]
Moreover, $G$ satisfies $CD(\infty,K)$ if
\[
\Gamma_2(f) \geq K\Gamma(f).
\]
\end{defi}
In the case of an $n$-dimensional Riemannian manifold whose Ricci curvature is bounded from below by $K$ the  curvature dimension inequality is a direct consequence of Bochner's identity. Even in more general settings where Bochner's identity is not available, the curvature dimension inequality has proven to be an important definition of curvature \cite{Bakry06, Lin10}.

However there are some problems with the curvature dimension inequality when one wants to prove the Li-Yau inequality for graphs. Indeed it turns out that a natural modification of the curvature dimension inequality is needed in order to prove the Li-Yau inequality.

\begin{defi}
A graph $G$ satisfies the  exponential curvature dimension inequality $CDE(n,K)$ if  for any vertex $x \in V$ and any positive function $f : V\to \R$ such that $\Delta f(x) < 0$ we have
\[  \Gamma_2(f)  - \Gamma\left(f, \frac{\Gamma(f)}{f}\right) \geq \frac{1}{n} (\Delta f)^2 + K \Gamma(f).\]

Moreover, $G$ satisfies the infinite dimensional exponential curvature dimension inequality $CDE(\infty,K)$ if
\[
 \Gamma_2(f)  - \Gamma\left(f, \frac{\Gamma(f)}{f}\right)  \geq K\Gamma(f).
\]
\end{defi}
From a general perspective, the exponential curvature dimension inequality is quite natural since it was shown in \cite{Bauer13} that it follows from the classical curvature dimension inequality in situations where the chain rule holds. Moreover on graphs (where the chain rule does not hold) the exponential curvature dimension inequality has some very nice properties compared to the curvature dimension inequality, see \cite{Bauer13} for more details.

\subsection{Gradient estimates and the Harnack inequality}

We recall some results in \cite{Bauer13} about the Li-Yau inequality (gradient estimate) and the corresponding Harnack inequality on graphs.
\begin{thm}\label{thm:weakballgre2}
Let $G(V,E)$ be a (finite or infinite) graph, $R > 0$, and fix $x_0 \in V$.
 Let $u: (0,\infty)\times V \to \R$ a positive solution to the heat equation $(\Delta -\partial_t) u(t,x) = 0$ if  $d(x,x_0) \leq 2R$.  If $G$ satisfies $CDE(n,0)$, then
 \begin{equation}\label{e:Li-Yau positive}\frac{\Gamma(\sqrt{u})}{u} -\frac{ \partial_t \sqrt{u}}{\sqrt{u}} < \frac{n}{2t} + \frac{n(1+D_\mu)D_m}{R}\end{equation} in the ball of radius $R$ around $x_0$.
\end{thm}
For general negative curvature lower bound and the Schr\"odinger operators with the potential $q,$ we have the following modification of Theorem \ref{thm:weakballgre2}.

\begin{thm}\label{thm:weakballgre} Let $G(V,E)$ be a (finite or infinite) graph, $R > 0$, and  $x_0 \in V$.
Let $u: (0,\infty)\times V \to \R$ a positive function such that $(\Delta-\partial_t-q) u(t,x) = 0$ if  $d(x,x_0) \leq 2R$, for some constant $q$.
If $G$ satisfies $CDE(n,-K)$ for some $K \geq 0$, then for any $0 < \rho < 1$
\[ (1-\rho) \frac{\Gamma(\sqrt{u})}{u} -\frac{ \partial_t \sqrt{u}}{\sqrt{u}} -\frac{q}{2}< \frac{n}{(1-\rho)2t} +  \frac{n(2+D_\mu)D_m}{(1-\rho)R} + \frac{Kn}{2\rho},\]
in the ball of radius $R$ around $x_0$.
\end{thm}
\begin{rem}Theorem \ref{thm:weakballgre2} and Theorem \ref{thm:weakballgre} are special cases of the main result in \cite{Bauer13}. In the most general case the potential $q$ may depend on the variables $x$ and $t$. For simplicity of exposition we restrict ourselves to the special case when $q$ is constant. However our results can easily be extended to the general case.
\end{rem}

On Riemannian manifolds~\cite{LiYau86}, a result similar to Theorem~\ref{thm:weakballgre} holds with $1/R^2$ instead of $1/R$ without any further assumptions.  In one of the key steps of the proof in the Riemannian case, the Laplacian comparison theorem is applied to the distance function. This together with the chain rule implies that one can find a cut-off function $\phi$ that satisfies \begin{equation}\label{eq:strong}\Delta \phi \geq -c(n) \frac{1 +R\sqrt{K}}{R^2},\end{equation}
and \begin{equation}\label{eq:strong1}\frac{|\nabla\phi|^2}{\phi}\leq\frac{c(n)}{R^2}\end{equation} where $c$ is a constant that only depends on the dimension $n$.

In contrast to manifolds, on graphs, one can only prove the Li-Yau inequality with $1/R$ (instead of $1/R^2$) without any further assumptions, see Theorem \ref{thm:weakballgre}.  The reason is that on graphs it is not clear that a cut-off function with similar properties always exists. However in order to prove the Li-Yau inequality with $1/R^2$ such a cut-off function is needed. This motivates the following definition.
\begin{defi}\label{def:strong_cutoff}
 Let $G(V,E)$ be a graph satisfying $CDE(n,-K)$ for some $K \geq 0$. We say that the function $\phi: V\to [0,1]$ is an \textit{$(c,R)$-strong} cut-off function centered at $x_0 \in V$ and supported on a set $S \subset V$ if $\phi(x_0 ) =1$, $\phi(x) = 0 $ if $x \not \in S$ and for any vertex $x \in S$
\begin{enumerate}
\item either $\phi(x) <\frac{c(1 + R\sqrt{K})}{2R^2}$,
\item or $\phi$ does not vanish in the immediate neighborhood of $x$ and
\[\phi^2(x) \Delta \frac{1}{\phi}(x) \leq D_m \frac{c(1 + R\sqrt{K})}{R^2} \;\mbox{ and }\; \phi^3(x) \Gamma\left(\frac{1}{\phi}\right)(x) \leq  D_m \frac{c}{R^2},\]
where the constant $c = c(n)$ only depends on the dimension $n$.
\end{enumerate}
\end{defi}

In case a strong cut-off function exists, one can prove the Li-Yau inequality with $1/R^2$.
\begin{thm} \label{thm:strongballgre} Let $G(V,E)$ be a (finite or infinite) graph satisfying $CDE(n,-K)$ for some $K \geq 0$.  Let $R > 0$ and $x_0 \in V$. Assume that $G$ has a $(c,R)$-strong cut-off function supported on $S \subset V$ and centered at $x_0$.  Let $u: (0,\infty)\times V \to \R$ be a positive function such that $(\Delta - \partial_t-q)u(t,x) = 0$ if  $x \in S$, for some constant  $q$.
Then for $0<\rho<  1$,
\begin{align*} &\left((1-\rho)\frac{\Gamma(\sqrt{u})}{u} -\frac{ \partial_t \sqrt{u}}{\sqrt{u}} -\frac{q}{2}\right)(t,x_0)  \\ <&  \frac{n}{2(1-\rho)t} + \frac{D_m cn}{2(1-\rho)R^2}\left(1 +R\sqrt{K}+\frac{n (D_\mu+1)^2}{4\rho (1-\rho)}\right) +\frac{Kn}{2\rho}.
\end{align*}
 \end{thm}

A corollary of the Li-Yau inequality is the following Harnack inequality that we will use together with the DGG Lemma to prove the heat kernel estimate in Section \ref{Section Applications DGG}.

\begin{thm}\label{thm:Harnack}Let $G(V,E)$ be a (finite or infinite) graph satisfying $CDE(n,-K)$ for some $K \geq 0$. If $u:(0,\infty)\times V \to \mathbb{R}$ is a positive solution to the equation $(\Delta-\partial_t-q )u(t,x)=0$ for some constant $q$ on the whole graph, then for any $0<\rho<1, 0<T_1\leq T_2,$ and  $x,y\in V,$ $$u(T_1,x)\leq u(T_2,y)\left(\frac{T_2}{T_1}\right)^{\frac{n}{1-\rho}}\exp\left(\left(\frac{Kn}{\rho}+q\right)(T_2-T_1)+\frac{4m_{\mathrm{max}}d^2(x,y)}{(1-\rho)(T_2-T_1)\mu_{\mathrm{min}}}\right).$$

\end{thm}
\subsection{Applications of the Li-Yau inequality}
In this section, we show several applications of the Li-Yau inequality on graphs.

As a first application of Li-Yau inequality in \cite{Bauer13}, we obtain Yau's Liouville theorem on positive harmonic functions on graphs satisfying $CDE(n,0)$.

\begin{thm}[Yau's Liouville theorem on graphs]
Let $G(V,E)$ be a graph satisfying the exponential curvature dimension inequality $CDE(n,0).$ Then any
positive harmonic function on $G$ is constant. In particular, bounded harmonic functions are constant.
\end{thm}
\begin{proof} For any time-independent positive harmonic function on $G,$ the Li-Yau gradient estimate
\eqref{e:Li-Yau positive} implies the Liouville theorem by letting $t\to\infty$ and $R\to \infty.$ The second part follows from the first one by considering the positive function $v(x) = u(x) -\inf u$.
\end{proof}
As we have seen, Yau's Liouville theorem follows directly from the Li-Yau inequality. Yau's Liouville theorem can also be proved by using the Moser iteration. This was initiated by Grigor'yan \cite{Grigoryan91} and Saloff-Coste \cite{SaloffCoste92} independently on Riemannian manifolds. Following their strategy, if we assume the volume doubling property and the Poincar\'e inequality, the Moser iteration \cite{Delmotte97} yields the Harnack inequality which will imply Yau's Liouville theorem on graphs. However it is difficult to compare these results since it is still unknown if the volume doubling property and the Poincar\'e inequality hold for graphs satisfying $CDE(n,0).$ Moreover, Saloff-Coste \cite{Saloff-Coste97} proved Yau's Liouville theorem for graphs satisfying certain conditions on the growth behavior of the volume of distance balls.

Our second application is an analogue to Cheng's Liouville theorem that any sublinear growth harmonic function on a Riemannian manifold with nonnegative Ricci curvature is constant. On general graphs satisfying $CDE(n,0)$, we can only prove the sub-square-root growth harmonic functions are constant, see below for the definition. However, if we further assume the existence of strong cut-off functions, then we obtain Cheng's Liouville theorem \cite{Yau75, Cheng80} for sublinear growth harmonic functions.

\begin{defi} For any
$R>0$, $x\in V$ and $u:B_R(x) \to\R,$ we define the oscillation of $u$ over the ball $B_R(x)$
by $$\osc_{B_R(x)}u:=\max_{B_R(x)}{u}-\min_{B_R(x)}u.$$ The function $u$ is
called of sub-square-root growth if
$$\max_{B_R(x)}|u|=o(R^{\frac12}) \mbox{ as}\ R\to\infty.$$ It is called of
sublinear growth if $$\max_{B_R(x)}|u|=o(R)\mbox{ as}\ R\to\infty.$$
\end{defi}
Clearly,
$u$ is of sub-square-root growth if and only if
$\osc_{B_R(x)}u=o(R^{\frac12})$ as $\ R\to\infty.$ Similarly, $u$ is of
sublinear growth if and only if $\osc_{B_R(x)}u=o(R)$ as $\ R\to\infty.$

\begin{thm}[Cheng's Liouville theorem on graphs]\label{thm:Cheng2}
Let $G=(V,E)$ be a graph satisfying the exponential curvature dimension inequality $CDE(n,0).$ Then any
sub-square-root growth harmonic function is constant. Furthermore,
if a strong cut-off function exists for any large ball, any sublinear growth harmonic
function is constant.
\end{thm}
\begin{proof}
Let $u$ be a sub-square-root growth harmonic function on $G,$ i.e.
for any $x\in V,$ $\osc_{B_{R}(x)}u=o(R^{\frac12})$ as $R\to\infty.$ For
any $R\geq 1,$ set $v:=u-\inf_{B_{2R}(x)}u+\epsilon$, for some $\epsilon > 0$. Then $v$ is a
positive harmonic function on $B_{2R}(x).$ Theorem \ref{thm:weakballgre2} implies the following gradient estimate for any time-independent positive harmonic functions $f$
$$\frac{\Gamma(\sqrt{f})}{f}(x)\leq \frac{C}{R},$$ for some constant $C$.
This yields
\begin{eqnarray*}
\Gamma(u)(x)
&=&\Gamma(v)(x)=\frac{1}{2m(x)}\sum_{y\sim x}\mu_{xy}(v(x)-v(y))^2\\
&=&\frac{1}{2m(x)}\sum_{y\sim x}\mu_{xy}\left(\frac{v(x)-v(y)}{\sqrt{v(x)}+\sqrt{v(y)}}\right)^2(\sqrt{v(x)}+\sqrt{v(y)})^2\\
&\leq & 4\Gamma(\sqrt{v})(x)(\osc_{B_{2R}(x)}u+\epsilon)\leq C\frac{(\osc_{B_{2R}(x)}u+\epsilon)^2}{R}
\end{eqnarray*} as $R\to\infty\mbox{ and }\epsilon\to 0$. Hence $\Gamma(u)(x)=0$ for any $x\in V.$ Thus, $u$ is a constant function.

If we assume the existence of strong cut-off functions, then the same argument as above using Theorem \ref{thm:strongballgre} yields the second assertion.
\end{proof}

As a further application of the Li-Yau inequality, we obtain an estimate for the greatest lower bound of the $\ell^2$-spectrum known as Cheng's eigenvalue estimate \cite{Cheng75b}.
\begin{thm}[Cheng's eigenvalue estimate on graphs]\label{thm:Cheng}
Let $G$ be a graph satisfying the exponential curvature dimension inequality $\CDE(n,-K)$ and let $\mu$ be the greatest lower bound for the $\ell^2$-spectrum of the graph Laplacian $\Delta.$ Then
$$\mu\leq Kn.$$
\end{thm}
\begin{proof}We note that Theorem 3.1 in \cite{Haeseler11} implies that if
$\lambda\leq \mu,$ then there exists a positive solution $u$ to the equation $$\Delta u=-\lambda u.$$ Moreover, for positive time-independent solutions to the equation $\Delta u=q u,$ the Li-Yau inequality Theorem \ref{thm:weakballgre} reduces to
$$(1-\rho)\frac{\Gamma(\sqrt u)}{u}-\frac{q}{2}\leq \frac{Kn}{2\rho}, \ \ \ \ \ \ \ \ \forall \rho\in (0,1).$$ Setting $q = -\lambda$ it follows that there exists a positive solution $u$ for $\Delta u = -\lambda u$ and $\lambda\leq \mu$ that satisfies
 \begin{equation}\label{cheng}
 (1-\rho)\frac{\Gamma(\sqrt u)}{u}+\frac{\lambda}{2}\leq \frac{Kn}{2\rho}.
 \end{equation}
Noting that $(1-\rho)\frac{\Gamma(\sqrt u)}{u} > 0$ and taking the limit $\rho\to 1$, we conclude that
 $$\mu\leq Kn,$$
since \eqref{cheng} is true for all $\lambda\leq\mu$.
\end{proof}

\section{Davies-Gaffney-Grigor'yan Lemma}\label{Section DGG}

In this section we give a proof of our main result, the DGG Lemma (Theorem \ref{thm:Davies}). In order to do that we need some preparation.

\begin{defi}
We say $u:[0,\infty)\times V \to \mathbb{R}$ solves the Dirichlet heat equation on $\Omega \subset V$ if
\begin{equation}\left\{ \begin{array}{r@{=}l@{\quad\forall}l}
\frac{\partial}{\partial t}u(t,x) & \Delta_{\Omega}u(t,x) & x\in\Omega, t\geq 0,\\
u(0,x) & f(x)& x\in \Omega
\\
u(t,x) & 0 & x \notin \Omega, t\geq 0. \end{array}\right.
\end{equation}
where $\Delta_{\Omega}$ is the Dirichlet Laplace Operator on $\Omega$, see for instance \cite{Coulhon98}. The Dirichlet heat kernel on $\Omega,$ $p_t(x,y,\Omega),$ is defined as the solution of the Dirichlet heat equation on $\Omega$ with the initial condition $f(x)=\frac{1}{m(y)}\delta_y(x).$ For a general initial data $f(x),$ the solution can be written as $$u(t,x) = \sum_{y\in \Omega} p_t(x,y,\Omega) f(y) m(y).$$ It is easy to see that
$$p_t(x,y,\Omega) = \sum_{k=1}^{|\Omega|}e^{-\lambda_k(\Omega)t}\phi_k(x)\phi_k(y),$$ where $\{\phi_k\}_{k=1}^{|\Omega|}$ is an orthonormal basis of eigenfunction of $\Delta_\Omega$ and $|\Omega|$ is the number of vertices in $\Omega$.
\end{defi}
\begin{defi}
 Let $\{\Omega_i\}_{i=1}^{\infty}$ be an exhaustion of $V$ by finite subsets, i.e. $$\Omega_1\subset \Omega_{2}\subset\cdots\subset \Omega_{i}\subset\cdots \subset\Omega,\ \ \mathrm{and}\ \ \cup_{i=1}^{\infty}\Omega_i=V.$$ Then we define the minimal heat kernel on $G$ by
$$p_t(x,y) := \lim_{i\to \infty} p_t(x,y,\Omega_i).$$ The maximum/minimum principle implies that the limit exists and that $p_t$ is minimal, i.e. for any other fundamental solution $q_t$ we have $q_t\geq p_t$. This indicates that the definition of the minimal heat kernel is independent of the choice of the exhaustion.
\end{defi}

First we prove a variant of the integral maximum principle on graphs which was introduced on Riemannian manifolds by Grigor'yan \cite{Grigoryan94}. For simplicity, we denote by $K_t(t,x)$ the partial derivative w.r.t. the variable $t$ of the $C^1$ function $K(t,x).$
\begin{lemma}[Integral maximum principle for finite subsets]\label{l:monotonicity Delmotte finite}
Let $u:[0,\infty)\times V\to\R$ solve the Dirichlet heat equation on $\Omega\subset V$ for some finite $\Omega$ and let $\mu_1=\mu_1(\Omega)$ be the first Dirichlet eigenvalue of $\Omega$.
Suppose that $K(t,x)$ is a nonnegative and nonincreasing $C^1-$function in $t$ and there exists a constant $\gamma\in [0,1]$ such that for any $t\geq 0,$ $x\sim y$ ($x,y\in V$)
\begin{eqnarray}\label{e:K(t,x) finite}&&\left(K(t,x)+K(t,y)-2(1-\gamma)\sqrt{K(t,x)K(t,y)}\right)^2\nonumber\\
&\leq& \left(\frac{1}{D_m}K_t(t,x)-2\gamma K(t,x)\right)\left(\frac{1}{D_m}K_t(t,y)-2\gamma K(t,y)\right),\end{eqnarray} then  $$e^{2(1-\gamma)\mu_1 t}I(t):= e^{2(1-\gamma)\mu_1 t}\sum_{x\in
\Omega}K(t,x)u^2(t,x)m(x),$$ is  nonincreasing in $t\in [0,\infty)$.
\end{lemma}
\begin{proof}
Direct calculation shows that \begin{eqnarray*}I'(t)&=&\sum_{x\in
\Omega}K_t(t,x) u^2(t,x)m(x)+2\sum_{x\in
\Omega} \sum_{y\in V}
\mu_{xy}K(t,x)u(t,x)(u(t,y)-u(t,x)),\\
&=&\sum_{x\in V}K_t(t,x)
u^2(t,x)m(x)+2\sum_{x\in V} \sum_{y\in V}
\mu_{xy}K(t,x)u(t,x)(u(t,y)-u(t,x)).\end{eqnarray*}
Using \eqref{a:on Dm}, $K_t(t,x)\leq 0$ and the symmetry of $\mu_{xy},$ we conclude that
$$\sum_{x\in V}u^2(t,x)K_t(t,x)m(x)\leq\frac{1}{2D_m}\sum_{x,y\in V} \mu_{xy}(u^2(t,x)K_t(t,x)+u^2(t,y)K_t(t,y))$$ and
\begin{eqnarray*}&&2\sum_{x,y\in V}\mu_{xy}K(t,x)u(t,x)(u(t,y)-u(t,x))\\&=&\sum_{x,y\in V}\mu_{xy}(u(t,y)-u(t,x))(u(t,x)K(t,x)-u(t,y)K(t,y)).\end{eqnarray*}
Hence
\begin{eqnarray*}I'(t)&\leq&\frac12\sum_{x,y\in V}\mu_{xy}\left(u^2(t,x)(\frac{1}{D_m}K_t(t,x)-2K(t,x))+\right.\\&&2u(t,x)u(t,y)(K(t,x)+K(t,y)) \left.+u^2(t,y)(\frac{1}{D_m}K_t(t,y)-2K(t,y))\right)\\
&\leq&-2(1-\gamma) \frac12\sum_{x,y\in V}\mu_{xy}\left(u(t,x)\sqrt{K(t,x)}-u(t,y)\sqrt{K(t,y)}\right)^2\\
&\leq& -2(1-\gamma)\mu_1 I(t),
\end{eqnarray*} where we have used \eqref{e:K(t,x) finite} for the quadratic expression in $u(t,x)$ and $u(t,y)$ in the second inequality. The last inequality follows from the Rayleigh quotient characterization of the first Dirichlet eigenvalue (see for instance \cite{Coulhon98})
$$\mu_1=\inf_{\substack{f: \mathrm{supp}(f)\subseteq \Omega,\\ f\not\equiv 0}}\frac{\frac{1}{2}\sum_{x,y\in N_1(\Omega)}\mu_{xy}(f(x)-f(y))^2}{\sum_{x\in\Omega}m(x)f^2(x)},$$ where $N_1(\Omega) = \{x\in V| d(x,\Omega)\leq 1\}$ is the $1$-neighborhood of $\Omega$ and $\mathrm{supp}(f)=\{x\in V: f(x)\neq 0\}.$
In fact, the choice $f(x) = u(t,x)\sqrt{K(t,x)}$ yields
\begin{eqnarray*}
\mu_1 &\leq & \frac{1}{2}\frac{\sum_{x,y\in N_1(\Omega)}\left(u(t,x)\sqrt{K(t,x)}-u(t,y)\sqrt{K(t,y)}\right)^2}{\sum_{x\in\Omega}m(x)u^2(t,x)K(t,x)}\\
&=& \frac{1}{2}\frac{\sum_{x,y\in V}\mu_{xy}\left(u(t,x)\sqrt{K(t,x)}-u(t,y)\sqrt{K(t,y)}\right)^2}{\sum_{x\in \Omega}m(x)u^2(t,x)K(t,x)}
\end{eqnarray*}
and hence
$$\mu_1 I(t) \leq \frac{1}{2}\sum_{x,y\in V}\mu_{xy}\left(u(t,x)\sqrt{K(t,x)}-u(t,y)\sqrt{K(t,y)}\right)^2$$
This proves the Lemma.
\end{proof}

Using an exhaustion argument as in \cite[Corollary~13.2]{Li12}, we can extend the integral maximum principle to the whole graph.
\begin{lemma}[Integral maximum principle]\label{l:monotonicity Delmotte}
Let $u(t,x) = \sum_{y\in V}p_t(x,y)f(x)$ solve the heat equation on $[0,\infty)\times V$ for $f\in \ell^p(V,m), p\in[1,\infty],$ and $\mu$ be the greatest lower bound for the $\ell^2$-spectrum of the graph Laplacian.
Suppose that $K(t,x)$ is a nonnegative and nonincreasing $C^1-$function function in $t$ and there exists a constant $\gamma\in [0,1]$ such that for any $t\geq 0,$ $x\sim y$ ($x,y\in V$)
\begin{eqnarray}\label{e:K(t,x)}&&\left(K(t,x)+K(t,y)-2(1-\gamma)\sqrt{K(t,x)K(t,y)}\right)^2\nonumber\\
&\leq& \left(\frac{1}{D_m}K_t(t,x)-2\gamma K(t,x)\right)\left(\frac{1}{D_m}K_t(t,y)-2\gamma K(t,y)\right),\end{eqnarray} then  $$e^{2(1-\gamma)\mu t}I(t):= e^{2(1-\gamma)\mu t}\sum_{x\in
V}K(t,x)u^2(t,x)m(x),$$ is  nonincreasing in $t\in [0,\infty)$.
\end{lemma}
\begin{rem}The special case $\gamma=1$ in the integral maximum principle was already obtained in \cite{Delmotte99} for continous time random walks and in \cite{Coulhon98,Coulhon05} for the discrete time random walk on graphs. However, the case $\gamma<1$ is of particular interest since it allows us to recover the exponential factor in the first Dirichlet eigenvalue. This exponential factor is very important (see also Remark \ref{Remark Exp}) and also appears in the DGG Lemma, Theorem \ref{thm:Davies}, and the heat-kernel estimates, Theorem \ref{thm:heat kernel estimate}.
\end{rem}
\begin{proof}We consider an exhaustion of $V$ by finite subsets $\{\Omega_i\}_{i=1}^{\infty}$. Let $u_i(t,x)$ be the solution of the Dirichlet heat equation on $\Omega_i$ with the initial condition $u_i(0,\cdot)=u|_{\Omega_i}(0,\cdot).$ By Lemma \ref{l:monotonicity Delmotte finite} for any finite $\Omega_i,$ $$t\mapsto e^{2(1-\gamma)\mu_1(\Omega_i) t}\sum_{x\in
\Omega_i}K(t,x)u_i^2(t,x)m(x)$$ is nonincreasing in $t$. Passing to the limit $i\to \infty$ we obtain the result since $\mu_1(\Omega_i)\to \mu$ and $u_i\to u.$
\end{proof}
\begin{rem}
By setting $K(t,x)=e^{2\eta(t,x)}=e^{2\eta(t,d(x))}$ where $d(x)=d(x,B),$ the distance function to some subset $B$ of $V,$ the equation \eqref{e:K(t,x)} is equivalent to
\begin{equation}\label{e:for chi}(\chi(\eta(t,x)-\eta(t,y))+\gamma)^2\leq \left(\frac{1}{D_m}\eta_t(t,x)-\gamma\right)\left(\frac{1}{D_m}\eta_t(t,y)-\gamma\right),\end{equation}
where $\chi(s)=\cosh(s)-1$.
\end{rem}
We want to use the integral maximum principle to prove the DGG Lemma. For that we need to find a non-trivial solution to \eqref{e:K(t,x)} or \eqref{e:for chi}. Recall that on Riemannian manifolds in order to apply the integral maximum principle, one needs to find a non-trivial solution to
\begin{equation}\label{manifoldeq}
\frac{\partial\eta}{\partial t}+\frac{1}{2}|\nabla\eta|^2\leq0.
\end{equation}
In this case $\eta=\frac{d^2}{2t}$ is a solution since the distance function $d$ satisfies $|\nabla d|\leq 1$ \cite{Grigoryan94}. Noting that $\chi(s)$ behaves like $\frac{s^2}{2}$ for small $s$ and setting $\gamma = 0, D_m=1$, one observes the obvious correspondence between \eqref{e:for chi} and \eqref{manifoldeq}. However, it is easy to see that $\frac{d^2}{2t}$ is not a solution to \eqref{e:for chi} for small $t$ (or more precisely $t/d$ small).
Still we want to find a non-trivial solution to \eqref{e:for chi} which behaves like $\frac{d^2}{2t}$ except for $t/d$ small.

In order to find such a solution of \eqref{e:for chi} we consider the Legendre associate
\begin{equation}\label{e:zeta def}\zeta(t,d) = \max_{\lambda\geq 0}\{d\lambda-\chi(\lambda)t\}\end{equation}
for any $t\geq 0$ and $d\geq 0$. Then
$$\zeta(t,d)=d\arcsinh\left(\frac{d}{t}\right)-\sqrt{d^2+t^2}+t,$$
and
\begin{equation}\label{e:partial t zeta}\frac{\partial }{\partial t}\zeta(t,d)=-\chi(\lambda(t,d))\end{equation}
where $\lambda(t,d)=\arcsinh(\frac{d}{t})$ is the value of $\lambda$ which attains the maximum in \eqref{e:zeta def}. The Legendre associate $\zeta$ was already used by Davies, Pang and Delmotte to obtain heat-kernel estimates \cite{Davies93,Pang93, Delmotte99}.

We have the following elementary lemma:
\begin{lemma}\label{l:convexity}
For any fixed $t\in (0,\infty),$ the function
$\zeta(t,d)=d\arcsinh\left(\frac{d}{t}\right)-\sqrt{d^2+t^2}+t$ is
increasing and convex in $d\in(0,\infty).$
\end{lemma}
\begin{proof}
Since $\zeta(t,d)=t\zeta(1,\frac{d}{t}),$ it suffices to show that
$\zeta(1,d)$ is convex. An elementary calculation yields the first and
second derivative with respect to $d$,
$$\zeta'(1,d)=\arcsinh(d)\geq 0,$$
$$\zeta''(1,d)=\frac{1}{\sqrt{d^2+1}}\geq 0.$$ This proves the
lemma.
\end{proof}
Moreover, one can show that \cite{Delmotte99}
\begin{equation}\label{e:zeta estimate}
\left\{\begin{array}{ll}
\zeta(t,d)\leq \frac{d^2}{2t}, & \text{for } t\geq0 \\
\zeta(t,d)\geq \sigma\arcsinh(\sigma^{-1})\frac{d^2}{2t},&\mbox{for}\; t\geq\sigma d.\\
\end{array} \right.
\end{equation}

The estimates in \eqref{e:zeta estimate} suggest that $\zeta$ is a good candidate for a solution of \eqref{e:for chi} since it behaves like $\frac{d^2}{2t}$ for $d/t$ small.  Indeed the next lemma shows that $\zeta$ is a solution of \eqref{e:for chi} up to the rescaling and shifting of the time.

\begin{lemma}\label{l:1-Lipschitz}
For any $0<\gamma\leq 1$ there exists a constant $\alpha(\gamma)\geq 1$ such that
\begin{equation}\label{e:K in lemma}K(t,x):=e^{2\zeta(\alpha D_m t+\frac12,d(x))}\end{equation} is nonincreasing in
$t\in [0,\infty)$ and satisfies \eqref{e:K(t,x)} where $d(x)$ is a distance function to some subset $B$ and $\zeta$ is defined in \eqref{e:zeta def}.
\end{lemma}
\begin{rem} One can consider an arbitrary time shift in \eqref{e:K in lemma}. However this does not give new insights and the choice $1/2$ leads to nice constants in our results.
\end{rem}
\begin{proof}
Set $\eta(t,x)=\zeta(\alpha D_m t+\frac12, d(x)).$ Using \eqref{e:partial t zeta}, we can rewrite equation \eqref{e:for chi} in the form
\begin{eqnarray}\label{e:simplified eq}(\chi(\eta(t,x)-\eta(t,y))+\gamma)^2&\leq& \left[\alpha \chi(\lambda(\alpha D_m t+\frac12,d(x)))+\gamma\right]\\&&\times \left[\alpha \chi(\lambda(\alpha D_m t+\frac12,d(y)))+\gamma\right].\nonumber \end{eqnarray}
Note that we have to prove \eqref{e:simplified eq} only for $x\sim y$. It is obvious that \eqref{e:simplified eq} is satisfied if $d(x) = d(y)$. By the symmetry of $x$ and $y$, we may assume w.l.o.g. that $d(x)>d(y)$. We distinguish the following two cases.

\textit{\bf Case 1.} $1\leq d(y)<d(x).$\\
First we observe that
$$0<\eta(t,x)-\eta(t,y)=\zeta(\alpha D_m t+\frac12, d(x))-\zeta(\alpha D_m t+\frac12, d(y))\leq \lambda(\alpha D_m t+\frac12,d(x)).$$
This can be seen as follows: Since $d(x) > d(y)$, it foloows from Lemma \ref{l:convexity} that
$$\zeta(\alpha D_m t + \frac{1}{2},d(x))\geq \zeta(\alpha D_m t + \frac{1}{2},d(y)).$$
By the definition of $\zeta$ we have
\begin{eqnarray*}
0 &\leq& \zeta(\alpha D_m t + \frac{1}{2},d(x))- \zeta(\alpha D_m t + \frac{1}{2},d(y))\\
&=& d(x)\lambda(x)-(\alpha D_m t + \frac{1}{2})\chi(\lambda(x)) - d(y)\lambda(y)+(\alpha D_m t + \frac{1}{2})\chi(\lambda(y))
\end{eqnarray*}
where $\lambda(x):=\lambda(\alpha D_m t + \frac{1}{2},d(x))$ and $\lambda(y):=\lambda(\alpha D_m t + \frac{1}{2},d(y))$ are the values of $\lambda$ that achieve the maximum in the definition of $\zeta$at the time $\alpha D_m t + \frac{1}{2}$ for $d=d(x)$ and $d=d(y)$ respectively.\\
Since
\begin{eqnarray*}
\zeta(\alpha D_m t + \frac{1}{2},d(y)) &=& \max_{\lambda\geq 0} \left\{d(y)\lambda -\chi(\lambda)(\alpha D_m t + \frac{1}{2})\right\}\\
&=& d(y)\lambda(y) -\chi(\lambda(y))(\alpha D_m t + \frac{1}{2})\\
&\geq & d(y)\lambda(x)-\chi(\lambda(x))(\alpha D_m t + \frac{1}{2})
\end{eqnarray*}
we have
\begin{eqnarray*}
0 &\leq & \zeta(\alpha D_m t + \frac{1}{2},d(x))-\zeta(\alpha D_m t + \frac{1}{2},d(y))\\
& \leq & (d(x)-d(y))\lambda(x) = \lambda(x)
\end{eqnarray*} where the last equality holds since $x\sim y$ and $d(x)>d(y)$.
Thus it suffices to find some constant $\alpha \geq 1$, such that
$$\chi(\lambda(\alpha D_m t+\frac12,d(x)))+\gamma\leq \alpha\chi(\lambda(\alpha D_m t+\frac12,d(y)))+\gamma$$ holds, or equivalently
$$\frac{\chi(\lambda(\alpha D_m t+\frac12,d(x)))}{\chi(\lambda(\alpha D_m t+\frac12,d(y)))}\leq \alpha.$$
By the discreteness of the distance function, i.e. $d(x)\in \N$, and $d(x)>d(y)\geq 1$, and $x\sim y,$ it folows that  $d(x)\leq 2 d(y).$ This yields
\begin{eqnarray*}\frac{\chi(\lambda(\alpha D_m t+\frac12,d(x)))}{\chi(\lambda(\alpha D_m t+\frac12,d(y)))}&=&
\frac{\sqrt{1+\frac{d(x)^2}{(\alpha D_m t+\frac12)^2}}-1}{\sqrt{1+\frac{d(y)^2}{(\alpha D_m t+\frac12)^2}}-1}\\&\leq& \frac{\sqrt{1+\frac{4d(y)^2}{(\alpha D_m t+\frac12)^2}}-1}{\sqrt{1+\frac{d(y)^2}{(\alpha D_m t+\frac12)^2}}-1}\\&\leq& 4.\end{eqnarray*} This proves the result in the first case by setting $\alpha \geq 4.$
Note that for this case we neither used the time shift nor assumed that $\gamma\neq 0$.
\\\textit{\bf Case 2.} $d(y)=0$ and $d(x)=1.$\\
In this case, equation \eqref{e:simplified eq} is equivalent to
\begin{equation}\label{chi3}(\chi (\zeta(\alpha D_m t+\frac12,1))+\gamma)^2\leq \gamma(\alpha \chi(\lambda(\alpha D_m t+\frac12,1))+\gamma).\end{equation} Note that \eqref{chi3} is false for $\gamma =0$. That is why we have to assume $\gamma>0$.
By definition $\zeta(\alpha D_m t+\frac12,1)\leq \lambda(\alpha D_m t+\frac12,1),$ which implies
$$(\chi (\zeta(\alpha D_m t+\frac12,1))+\gamma)^2\leq (\chi(\lambda(\alpha D_m t+\frac12,1))+\gamma)^2.$$
Moreover since we introduced the time shift $1/2,$ we have $\chi(\lambda(\alpha D_m t+\frac12,1))=\sqrt{1+\frac{1}{(\alpha D_m t+\frac12)^2}}-1\leq \sqrt5-1$ for any $t\geq 0.$
Choosing $\alpha \geq\frac{\sqrt5-1}{\gamma}+2,$ we have
\begin{eqnarray*}
\left(\chi(\zeta(\alpha D_m t + \frac{1}{2},1))+\gamma\right)^2 &\leq & \left(\chi(\lambda(\alpha D_m t+\frac12,1))+\gamma\right)^2\\
&\leq & \gamma \left(\alpha \chi(\lambda(\alpha D_m t+\frac12,1))+\gamma\right).
\end{eqnarray*}
This proves the result in the second case. Hence the lemma follows by choosing $\alpha(\gamma) =\max\{4,\frac{\sqrt5-1}{\gamma}+2\}.$
\end{proof}

\begin{rem} Unfortunately, we cannot prove the lemma for $\gamma=0$. The reason is that in our proof the constant  $\alpha(\gamma)\to \infty$ as $\gamma\to 0$.
\end{rem}

Now, we are ready to prove the DGG Lemma for graphs.

\begin{proof}[Proof of Theorem \ref{thm:Davies}] For infinite subsets $B_1$ and $B_2,$ we can take an exhaustion by finite subsets. Since the estimates \eqref{e:Davies estimate} and \eqref{e: Davies gamma=1} are stable by passing to the limit of the exhaustion, it suffices to prove the theorem for finite subsets $B_1$ and $B_2.$ For the case $0<\gamma<1,$ we set
\begin{eqnarray*}f_i(t,x)&:=&\sum_{y\in B_i}p_t(x,y)m(y),\\
K_i(t,x)&:=&e^{2\zeta(\alpha D_m t+\frac12,d(x,B_i))},\ \ \ \ \ i=1,2\end{eqnarray*} where $\alpha=\alpha(\gamma)$ is the constant in Lemma \ref{l:1-Lipschitz}.
Lemma \ref{l:monotonicity Delmotte} and Lemma \ref{l:1-Lipschitz}
imply that for any $t\geq 0,$
$$e^{2(1-\gamma)\mu t}\sum_{x\in V}K_i(t,x)f_i^2(t,x)m(x)\leq \sum_{x\in
V}K_i(0,x)f_i^2(0,x)m(x).$$ Note that
$$f_i(0,x)=\sum_{y\in B_i}p_0(x,y)m(y)=\mathds{1}_{B_i}(x)$$ where
$\mathds{1}_{B_i}$ is the characterization function of $B_i,$ $i=1,2$. This yields
that
$$\sum_{x\in V}K_i(0,x)f_i^2(0,x)m(x)=m(B_i).$$ Hence
\begin{equation}\label{Davieseq}
e^{2(1-\gamma)\mu t}\sum_{x\in V}K_i(t,x)f_i^2(t,x)m(x)\leq m(B_i)\ \ {\rm for\ all \ } t\geq 0.
\end{equation}
By Lemma \ref{l:convexity}, $\zeta(t,\cdot)$ is increasing and convex in $d$.
Applying Jensen's inequality together with the triangle inequality implies that for any $t\geq 0$ and $x\in V$
\begin{eqnarray*}
\zeta\left(\alpha D_m t+\frac12,\frac{d(B_1,B_2)}{2}\right)
&\leq & \zeta\left(\alpha D_m t+\frac12,\frac{d(x,B_1)+d(x,B_2)}{2}\right)\\&\leq & \frac{1}{2}\left[\zeta(\alpha D_m t+\frac12,d(x,B_1))+\zeta(\alpha D_m t+\frac12,d(x,B_2))\right].
\end{eqnarray*}
This yields
$$e^{\zeta(2\alpha D_m t+1,d(B_1,B_2))}\leq \sqrt{K_1(t,x)K_2(t,x)}$$
since $\zeta(t,d)=t\zeta(1,\frac{d}{t}),$ and thus $$\zeta(2\alpha D_m t+1,d(B_1,B_2))=2\zeta\left(\alpha D_m t+\frac12,\frac{d(B_1,B_2)}{2}\right).$$
Hence
\begin{eqnarray*}
&&\sum_{x\in
V}e^{\zeta(2\alpha D_m t+1,d(B_1,B_2))}f_1(t,x)f_2(t,x)m(x)\\&\leq&\sum_{x\in
V}\sqrt{K_1(t,x)K_2(t,x)}f_1(t,x)f_2(t,x)m(x)\\
&\leq& \left(\sum_{x\in
V}K_1(t,x)f_1^2(t,x)m(x)\right)^{\frac12}\left(\sum_{x\in
V}K_2(t,x)f_2^2(t,x)m(x)\right)^{\frac12}\\
&\leq &e^{-2(1-\gamma)\mu t}\sqrt{m(B_1)m(B_2)},
\end{eqnarray*}
where we used Cauchy-Schwarz in the second and \eqref{Davieseq} in the third inequality.\\
In addition, by the semigroup property, the left-hand side can be
written as
\begin{eqnarray*}
&& \sum_{x\in V}e^{\zeta(2\alpha D_m t+1,d(B_1,B_2))}f_1(t,x)f_2(t,x)m(x)\\
&=& e^{\zeta(2\alpha D_m t+1,d(B_1,B_2))}\sum_{x\in V}\sum_{y\in
B_1}\sum_{z\in B_2}p_t(x,y)p_t(x,z)m(x)m(y)m(z)\\
&=& e^{\zeta(2\alpha D_m t+1,d(B_1,B_2))}\sum_{y\in B_1}\sum_{z\in
B_2}\left(\sum_{x\in V}p_t(y,x)p_t(x,z)m(x)\right)m(y)m(z)\\
&=& e^{\zeta(2\alpha D_m t+1,d(B_1,B_2))}\sum_{y\in B_1}\sum_{z\in
B_2}p_{2t}(y,z)m(y)m(z).
\end{eqnarray*}
Combining these results and rescaling the time by the factor $\frac12$, the result follows.

For the case $\gamma = 1$, i.e. the case when we do not have the exponential factor in $\mu$, we do not need to rescale and shift the time. In this case one can show that
$$\sum_{x\in V}\widetilde{K_i}(t,x)f_i^2(t,x)m(x),\,i=1,2$$ is non-increasing in $t\in [0,\infty)$ where
$$\widetilde{K_i}(t,x):=e^{\zeta(D_m t,d(x,B_i))}.$$
The same argument yields the result in this case.
\end{proof}

Using the properties of $\zeta,$ \eqref{e:zeta estimate}, we obtain the following corollary.
\begin{coro}\label{c:Davies lemma}
Let $p_t(x,y)$ be the minimal heat
kernel of the graph $G$ and $\beta>0.$ Then for any $0<\gamma< 1,$ there exist a constant $C_3(\gamma,\beta,D_m)$ such that for any subsets $B_1,B_2\subset G, \ t\geq  \beta d(B_1,B_2)\vee 1,$
\begin{eqnarray}\label{e:Davies estimate in coro}
&&\sum_{x\in B_1}\sum_{y\in B_2}p_t(x,y)m(x)m(y)\nonumber\\&\leq& e^{-(1-\gamma)\mu t}
\sqrt{m(B_1)m(B_2)}\exp\left(-C_3\frac{d^2(B_1,B_2)}{4t}\right).
\end{eqnarray}
Moreover, for the case $\gamma=1,$ we have for any $t\geq \beta d(B_1,B_2),$
\begin{equation*}
\sum_{x\in B_1}\sum_{y\in B_2}p_t(x,y)m(x)m(y)\leq
\sqrt{m(B_1)m(B_2)}\exp\left(-C\frac{d^2(B_1,B_2)}{4t}\right),
\end{equation*} where $C= C(\beta,D_m)=\beta\arcsinh(\frac{1}{D_m\beta}).$
\end{coro}

\begin{rem}\label{Remark Exp}
\begin{enumerate}[(a)]\item For $\gamma=1$, a similar result was obtained in \cite{Coulhon98, Coulhon05} for the discrete time heat kernel on graphs.
\item
This result shows the importance of the case $\gamma<1$. Although we obtain the right constant in the exponential in $d^2/t$ for $\gamma=1$ and $t$ large (note that $\sigma\arcsinh(\sigma^{-1})\to 1$ as
$\sigma\to\infty$) we cannot recover the exponential factor in $\mu$. In contrast, for $\gamma<1$ we lose some constant in the exponential in $d^2/t$ but we are able to recover the exponential factor in $\mu$.  This is important since for large $t$ the right hand side for $\gamma<1$ goes to zero whereas the right hand side for $\gamma=1$ converges to a positive constant.
\item The constant $C_3$ in this corollary can be chosen as $$C_3=C_3(\gamma,\beta,D_m) = \frac{2 \alpha D_m\beta\arcsinh\left(\frac{1}{\alpha D_m\beta}\right)}{\alpha D_m+1},$$ where $\alpha=\alpha(\gamma)$ is the constant in Lemma \ref{l:1-Lipschitz}.
\end{enumerate}
\end{rem}
In particular, we have the following explicit estimate.
\begin{coro}
Let $G$ be an infinite graph, $D_m=1$ and $p_t(x,y)$ be the minimal heat
kernel of $G$. Then for any subsets $B_1,B_2\subset G$ and  $t\geq d(B_1,B_2)\vee 1,$
\begin{eqnarray}\label{e:Davies estimate in coro2}
&&\sum_{x\in B_1}\sum_{y\in B_2}p_t(x,y)m(x)m(y)\nonumber\\&\leq& e^{-(\frac{3-\sqrt{5}}{2})\mu t}
\sqrt{m(B_1)m(B_2)}\exp\left(-\frac{8 \arcsinh(1/4)}{5}\frac{d^2(B_1,B_2)}{4t}\right).
\end{eqnarray}
\end{coro}

\section{Applications of the Davies-Gaffney-Grigor'yan Lemma}\label{Section Applications DGG}
\subsection{Heat kernel estimates}
Combining the Harnack inequality, Theorem \ref{thm:Harnack}, and the DGG Lemma, Theorem \ref{thm:Davies}, we can now prove the heat kernel estimates for graphs satisfying the exponential curvature dimension inequality.

\begin{proof}[Proof of Theorem \ref{thm:heat kernel estimate}]
Since we have the DGG Lemma on graphs, we can closely follow the standard proof in the continuous case, see \cite{Li12}. Fix $x,y\in V$ and $\delta>0.$ Applying the Harnack inequality, Theorem \ref{thm:Harnack}, to the heat kernel $p_t(x,y)$ with $T_1=t$ and $T_2=(1+\delta )t$ yields
\begin{eqnarray*}p_t(x,y)&\leq& p_{(1+\delta)t}(x',y)(1+\delta)^{C_4}\exp\left(C_5\delta t+\frac{4m_{\mathrm{max}}d^2(x,x')}{(1-\rho)\delta t \mu_{\mathrm{min}}}\right)\\
&\leq&p_{(1+\delta)t}(x',y) (1+\delta)^{C_4}\exp\left(C_5\delta t+\frac{4m_{\mathrm{max}}}{(1-\rho)\delta  \mu_{\mathrm{min}}}\right),\ \ \ \ \ \forall x'\in B_x(\sqrt t),\end{eqnarray*} where $C_4=\frac{n}{1-\rho}, C_5=\frac{Kn}{\rho}.$
Summing over all $x'\in B_x(\sqrt t)$ yields
\begin{equation}\label{e:davies to heat 1}
m (B_x(\sqrt t)) p_t(x,y) \leq (1+\delta)^{C_4}\exp\left(C_5\delta t+\frac{4m_{\mathrm{max}}}{(1-\rho)\delta  \mu_{\mathrm{min}}}\right)\sum_{x'\in B_x(\sqrt t)}m(x')p_{(1+\delta)t}(x',y).
\end{equation}
Using again the Harnack inequality for the following positive solution to the heat equation, $$h(y,s)=\sum_{x'\in B_x(\sqrt t)}m(x')p_s(x',y),$$
and setting $T_1=(1+\delta)t, T_2=(1+2\delta)t$ yields
\begin{eqnarray*}
&& m(B_y(\sqrt{t}))\sum_{x'\in B_x(\sqrt t)}m(x')p_{(1+\delta)t}(x',y)\leq
\left(\frac{1+2\delta}{1+\delta}\right)^{C_4}\exp\left(C_5\delta t+\frac{4m_{\mathrm{max}}}{(1-\rho)\delta  \mu_{\mathrm{min}}}\right)\\&&\ \ \ \ \  \ \ \ \ \qquad \qquad \times \sum_{x'\in B_x(\sqrt t)}\sum_{y'\in B_y(\sqrt{t})}m(x')m(y')p_{(1+2\delta)t}(x',y').
\end{eqnarray*}
Together with \eqref{e:davies to heat 1} this yields
\begin{eqnarray*}
p_t(x,y)&\leq& (1+2\delta)^{C_4}\exp\left(2C_5\delta t+\frac{8m_{\mathrm{max}}}{(1-\rho)\delta  \mu_{\mathrm{min}}}\right)\\
&&\times \frac{1}{m (B_x(\sqrt t))m (B_y(\sqrt t))}\sum_{x'\in B_x(\sqrt t)}\sum_{y'\in B_y(\sqrt{t})}m(x')m(y')p_{(1+2\delta)t}(x',y').
\end{eqnarray*}

For $t\geq \beta d(x,y)\vee 1\geq \frac{1}{1+2\delta}(\beta d(x,y)\vee 1)\geq \frac{1}{1+2\delta}(\beta d(B_x(\sqrt t),B_y(\sqrt t))\vee 1)$  Corollary \ref{c:Davies lemma} implies that there exists $C_3(\gamma,\beta,D_m)$ such that
\begin{eqnarray*}
&&\sum_{x'\in B_x(\sqrt t)}\sum_{y'\in B_y(\sqrt t)}p_{(1+2\delta)t}(x',y')m(x')m(y')\\
&\leq & \exp(-(1-\gamma)\mu(1+2\delta)t)\sqrt{m(B_x(\sqrt t))m(B_y(\sqrt t))}\exp\left(-C_3\frac{d^2(B_x(\sqrt t),B_y(\sqrt t))}{4(1+2\delta)t}\right).
\end{eqnarray*}
Using this we obtain
\begin{eqnarray*}
p_t(x,y)&\leq& (1+2\delta)^{C_4}\exp\left(2C_5\delta t+\frac{8m_{\mathrm{max}}}{(1-\rho)\delta  \mu_{\mathrm{min}}}\right) \frac{1}{\sqrt{m (B_x(\sqrt t))m (B_y(\sqrt t))}}\\
&&\times\exp\left(-(1-\gamma)\mu(1+2\delta)t-C_3\frac{d^2(B_x(\sqrt t),B_y(\sqrt t))}{4(1+2\delta)t}\right).
\end{eqnarray*}

We observe the following
\[d(B_x(\sqrt t),B_y(\sqrt t))=\left\{\begin{array}{ll}
0, & \mathrm{if}\ d(x,y)\leq 2\lfloor\sqrt t\rfloor,\\
d(x,y)-2\lfloor \sqrt t\rfloor, & \mathrm{if}\ d(x,y)> 2\lfloor\sqrt t\rfloor,\\
\end{array} \right.\] where $\lfloor\sqrt t\rfloor$ is the greatest integer less than or equal to $\sqrt t.$
It follows that
\[d(B_x(\sqrt t),B_y(\sqrt t))\geq\left\{\begin{array}{ll}
0, & \mathrm{if}\ d(x,y)\leq 2\lfloor\sqrt t\rfloor,\\
d(x,y)-2\sqrt t, & \mathrm{if}\ d(x,y)> 2\lfloor\sqrt t\rfloor.\\
\end{array} \right.\]
Hence we have
$$-\frac{d^2(B_x(\sqrt t),B_y(\sqrt t))}{4(1+2\delta)t}=0\leq 1-\frac{d^2(x,y)}{4(1+4\delta)t},\ \ \ \mathrm{if}\ d(x,y)\leq 2\lfloor \sqrt t\rfloor$$
and
\begin{eqnarray*}
-\frac{d^2(B_x(\sqrt t),B_y(\sqrt t))}{4(1+2\delta)t}&\leq& -\frac{(d(x,y)-2\sqrt t)^2}{4(1+2\delta)t}\\
&\leq & -\frac{d^2(x,y)}{4(1+4\delta)t}+\frac{1}{2\delta},\ \ \ \mathrm{if}\ d(x,y)>2\lfloor\sqrt t\rfloor.
\end{eqnarray*}

Combining all above there exists a constant $C=e^{C_3}$ such that
\begin{eqnarray*}
p_t(x,y) &\leq & C(1+2\delta)^{C_4}\exp\left(2C_5\delta t+\frac{8m_{\mathrm{max}}}{(1-\rho)\delta  \mu_{\mathrm{min}}}+\frac{C_3}{2\delta}-(1-\gamma)\mu(1+2\delta)t\right)\\
&&\times \frac{1}{\sqrt{m (B_x(\sqrt t))m (B_y(\sqrt t))}}\exp\left(-C_3\frac{d^2(x,y)}{4(1+4\delta)t}\right).
\end{eqnarray*}

We consider two cases.

\textit{\bf Case 1.} $C_5t\geq 1.$ We choose $2\delta=\frac{\epsilon}{\sqrt{C_5 t}}.$ This yields
\begin{eqnarray*}p_t(x,y)&\leq& C(1+\frac{\epsilon}{\sqrt{C_5 t}})^{C_4}\exp\left[\sqrt{C_5t}\left(\frac{C_3}{\epsilon}+\frac{16m_{\mathrm{max}}}{(1-\rho)  \mu_{\mathrm{min}}\epsilon}+\epsilon\right)\right]\\
&&\times
\frac{1}{\sqrt{m(B_x(\sqrt t))m(B_y(\sqrt t))}}\exp\left(-(1-\gamma)\mu t-\frac{C_3d^2(x,y)}{4(1+\frac{2\epsilon}{\sqrt{C_5t}})t}\right).
\end{eqnarray*}
Hence
\begin{eqnarray*}p_t(x,y)&\leq& C(1+\epsilon)^{C_4}\exp\left[\sqrt{C_5t}\left(\frac{C_3}{\epsilon}+\frac{16m_{\mathrm{max}}}{(1-\rho)  \mu_{\mathrm{min}}\epsilon}+\epsilon\right)\right]\\
&&\times
\frac{1}{\sqrt{m(B_x(\sqrt t))m(B_y(\sqrt t))}}\exp\left(-(1-\gamma)\mu t-\frac{C_3d^2(x,y)}{4(1+2\epsilon)t}\right).
\end{eqnarray*}

\textit{\bf Case 2.} $C_5t<1.$ We choose $2\delta=\epsilon.$ This yields
\begin{eqnarray*}p_t(x,y)&\leq& C(1+\epsilon)^{C_4}\exp\left(\epsilon \sqrt{C_5t}+\frac{C_3}{\epsilon}+\frac{16m_{\mathrm{max}}}{(1-\rho)  \mu_{\mathrm{min}}\epsilon}\right)\\
&&\times
\frac{1}{\sqrt{m (B_x(\sqrt t))m(B_y(\sqrt t))}}\exp\left(-(1-\gamma)\mu t-\frac{C_3d^2(x,y)}{4(1+2\epsilon)t}\right).
\end{eqnarray*}
Choosing some fixed value for $\rho\in (0,1)$, say for instance $\rho = 1/2$ completes the proof.
\end{proof}


As an easy corollary of Theorem \ref{thm:heat kernel estimate} we obtain heat-kernel estimate for finite graphs, see  \cite{Lee95} for a similar result on manifolds. For a finite graph $G$ on $N$ vertices, we order the eigenvalues of $G$ in the non-decreasing way: $0=\lambda_1 <\lambda_2\leq \ldots \leq \lambda_{N}.$ Note that the heat kernel converges in this estimate to $\frac{1}{V}$ in an explicit way, where $V=m(G)$ is the volume of the whole graph.
\begin{coro}\label{coro: heat kernel estimate} Let $G$ be a finite graph on $N$ vertices and $D:= \max_{x,y\in V} d(x,y)$ its diameter. If $G$ satisfies the exponential curvature dimension inequality $\CDE(n,-K),$ then for all $0<\rho<1$,
$$\left|p_t(x,y)-\frac{1}{V}\right|\leq \frac{1}{V}\left(C_1\exp\left(C_2\sqrt{Kn}D\right)-1\right) \exp(\lambda_2D^2-\lambda_2t)$$ for any $t\geq D^2,$ where $\lambda_2$ is the smallest nontrivial eigenvalue of $G$ and the constants $C_1$ and $C_2$ are the same as in Theorem \ref{thm:heat kernel estimate}.
\end{coro}
\begin{proof}We follow the proof on manifolds \cite{Lee95}.
For the heat kernel we have the well-known eigenfunction expansion \cite{Chung97} $$p_t(x,y)=\sum_{i=1}^{N}e^{-\lambda_i t}\phi_i(x)\phi_i(y),$$ where $\{\phi_i\}_{i=1}^{N}$ is a complete set of orthonormal eigenfunctions of the Laplacian, i.e.
$$\sum_{x\in V}m(x)\phi_i(x)\phi_j(x)=\delta_{ij}.$$
Since the graph $G$ is finite, $\lambda_1=0$ and $\phi_1=\frac{1}{\sqrt V}.$ For simplicity we define  $h_t(x,y):=p_t(x,y)-\frac{1}{V}$. Then $h_t$ is given by
\begin{eqnarray*}h_t(x,y)&=&\sum_{i=2}^{N}e^{-\lambda_i t}\phi_i(x)\phi_i(y)\\
&=&e^{-\lambda_2 t}\sum_{i=2}^{N} e^{(\lambda_2-\lambda_i)t}\phi_i(x)\phi_i(y).
\end{eqnarray*}
Multiplying through by $e^{\lambda_2t}$ we see that  $h_t(x,x)e^{\lambda_2 t}$ is nonincreasing in $t.$
Using the heat kernel estimate Theorem \ref{thm:heat kernel estimate} for $x=y$ and $t=D^2,$ we get
$$p_{D^2}(x,x)\leq \frac{C_1}{V}\exp\left(C_2\sqrt{Kn}D\right).$$ Since
$h_t(x,x)e^{\lambda_2 t}$ is nonincreasing in $t$, this yields
\begin{eqnarray*}h_t(x,x)e^{\lambda_2 t}&\leq& h_{D^2}(x,x) e^{\lambda_2 D^2}\\
&\leq&\frac{1}{V}\left(C_1\exp\left(C_2\sqrt{Kn}D\right)-1\right)e^{\lambda_2 D^2}, \ \ \ \forall\ t\geq D^2.
\end{eqnarray*}
Using Cauchy-Schwartz inequality, we get
\begin{eqnarray*}h_t(x,y)^2&=&\left(\sum_{i=2}^{N}e^{-\lambda_it}\phi_i(x)\phi_i(y)\right)^2\\
&\leq&\left(\sum_{i=2}^{N}e^{-\lambda_i t}\phi_i^2(x)\right)\left(\sum_{i=2}^{N}e^{-\lambda_{i}t}\phi_i^2(y)\right)\\
&=&h_t(x,x)h_t(y,y).
\end{eqnarray*}
This implies that
$$|h_t(x,y)|\leq \frac{1}{V}\left(C_1 \exp\left(C_2\sqrt{Kn}D\right)-1\right)e^{\lambda_2D^2-\lambda_2t}.$$
This proves the corollary.
\end{proof}
\subsection{Eigenvalue estimates}
For a compact Riemannian manifold $M$, Chung, Grigor'yan and Yau \cite{Chung-Grigoryan-Yau96} showed by using the DGG Lemma \ref{DGGManifold} that the smallest positive Neumann eigenvalue of the Laplacian satisfies
\begin{equation}\label{CGY} \lambda_2 \leq \frac{C_1}{d(X,Y)^2}\left(\log\frac{C_2\mathrm{vol}(M)}{\sqrt{\mathrm{vol} (X)\mathrm{vol} (Y)}}\right)^2,\end{equation} where $X,Y$ are two disjoint subsets of $M$. Later on the constants $C_1$ and $C_2$ were improved  \cite{Chung-Grigoryan-Yau97, Bobkov97, friedman00} by other methods. Moreover in their papers \cite{Chung-Grigoryan-Yau96,Chung-Grigoryan-Yau97} Chung, Grigor'yan and Yau obtained similar but weaker estimates for graphs that are of order $1/d$ instead of $1/d^2$. It was an open question whether the eigenvalue estimates on graphs can be improved and similar results to those on Riemannian manifolds can be obtained. Friedman and Tillich \cite{FriedmanTillich04} observed that this improvement is indeed possible. Their strategy was to use the strong estimates on manifolds and transfer them in a clever way to the graph setting. Here as an application of the DGG Lemma, we give a direct proof of the $1/d^2$ estimate for graphs that is purely discrete and does not use the results on manifolds.  However we have to point out that our proof that follows \cite{Chung-Grigoryan-Yau96} yields worse constants than the results in \cite{FriedmanTillich04}. We also note that higher order eigenvalue estimates similar to \eqref{CGY} are known on manifolds and graphs, \cite{Chung-Grigoryan-Yau96,Chung-Grigoryan-Yau97, FriedmanTillich04}.
\begin{thm}\label{thm:eigenvalue estimates} Let $G$ be a finite graph on $N$ vertices and order the eigenvalues of $G$ in the nondecreasing way: $0=\lambda_1 <\lambda_2\leq \ldots \leq \lambda_{N}.$ Let $A_1,A_2,\cdots, A_k$ be $k$ disjoint subset on $G$ and  $$\delta:=\min_{i\neq j}d(A_i,A_j).$$ Then
\begin{equation}\label{e:our eigenvalue estimate}\lambda_k\leq \frac{D_m}{\delta}\max_{i\neq j}\frac{\log \frac{2m(V)}{\sqrt{m(A_i)m(A_j)}}}{h \left(\frac{2}{\delta}\log \frac{2m(V)}{\sqrt{m(A_i)m(A_j)}}\right)},\end{equation}
\end{thm} where $h(t)$ is the inverse function of $\zeta(t,1).$
\begin{rem}\label{rem: eigenvalue estimates}\begin{enumerate}[(a)]
\item Note that in the Riemannian case, $\zeta(t,d)$ corresponds to $\frac{d^2}{2t},$ and $h(t)$ to $\frac{1}{2t}.$
\item Using the properties \eqref{e:zeta estimate} of $\zeta(t,1)$ it is easy to see that
$$h(t) \geq \sigma \arcsinh(\sigma^{-1}) \frac{1}{2t},\ \  \text{ for }  t \leq \frac{1}{2} \arcsinh(\sigma^{-1}).$$ Thus, if we choose
\begin{equation}\label{e: sigma}\sigma = \left(\sinh\left(\frac{4}{\delta}\max_{i\neq j}\log\frac{2m(V)}{\sqrt{m(A_i)m(A_j)}}\right)\right)^{-1},\end{equation} then
  \begin{equation}\label{Eigenvalue estimate}\lambda_k\leq \frac{4D_m}{\sigma\arcsinh(\sigma^{-1})\delta^2}\max_{i\neq j}\left(\log \frac{2m(V)}{\sqrt{m(A_i)m(A_j)}}\right)^2.\end{equation} Note that if
$$\frac{4}{\delta}\max_{i\neq j}\log\frac{2m(V)}{\sqrt{m(A_i)m(A_j)}}<<1,$$ then we can choose $\sigma$ such that $\sigma\arcsinh(\sigma^{-1})\approx 1$. Moreover, since $\delta\geq1$ we can always define $\sigma$ independently of $\delta$ by replacing $\delta$ by $1$ in \eqref{e: sigma}.
\end{enumerate}
\end{rem}

\begin{proof}[Proof of Theorem \ref{thm:eigenvalue estimates}] Using the DGG Lemma \ref{thm:Davies} we can follow closely the proof of \cite[Theorem~1.1]{Chung-Grigoryan-Yau96} for Riemannian manifolds, see also \cite[Theroem~4.1]{Grigoryan99}.
Let $\{\phi_i\}_{i=1}^N$ be an orthonormal basis of $\ell^2(V,m)$ consisting of eigenfunctions pertaining to the eigenvalues $\{\lambda_i\}_{i=1}^N$ of the Laplacian $\Delta$.

For convenience, we divide the proof into two cases:

{\bf Case 1.} $k=2.$
The characteristic functions $\mathds{1}_{A_1}$ and $\mathds{1}_{A_2}$ can be expressed as (generalized Fourier expansion)
$$\mathds{1}_{A_1}=\sum_{i=1}^Na_i \phi_i, \ \text{ and } \mathds{1}_{A_2}=\sum_{i=1}^N b_i \phi_i,$$ where $a_i=(\mathds{1}_{A_1},\phi_i)_{\ell^2(V,m)}$ and $b_i=(\mathds{1}_{A_2},\phi_i)_{\ell^2(V,m)}.$ Obviously, $$\sum_{i=1}^Na_i^2=\|\mathds{1}_{A_1}\|_{\ell^2(V,m)}=m(A_1), \ \ \sum_{i=1}^N b_i^2=m(A_2).$$ In addition, by $\phi_1=\frac{1}{\sqrt{m(V)}},$ $$a_1=\frac{m(A)}{\sqrt{m(V)}}, \ \ b_1=\frac{m(B)}{\sqrt{m(V)}}.$$
Since $p_t(x,y)=\sum_{i=1}^Ne^{-\lambda_i t}\phi_i(x)\phi_i(y)$,
\begin{eqnarray*}\sum_{x\in A_1}\sum_{y\in A_2}p_t(x,y)m(x)m(y)&=&a_1b_1+\sum_{i=2}^Ne^{-\lambda_i t}a_ib_i\\
&\geq& a_1b_1-e^{-\lambda_2 t}\left(\sum_{i=2}^N a_i^2\right)^{\frac12}\left(\sum_{i=2}^N b_i^2\right)^{\frac12}\\
&\geq& \frac{m(A_1)m(A_2)}{m(V)}-e^{-\lambda_2 t}\sqrt{m(A_1)m(A_2)}.
\end{eqnarray*}
By \eqref{e: Davies gamma=1} in DGG Lemma, we have
$$e^{-t\lambda_2}\geq \frac{\sqrt{m(A_1)m(A_2)}}{m(V)}-e^{-\frac{1}{2}\zeta(D_mt, \delta)},$$ where $\delta=d(A_1,A_2).$
Note that for any $d>0,$ $\zeta(t,d)$ is strictly nonincreasing in $t$, $\zeta(t,d)\to \infty,$ as $t\to 0$ and $\zeta(t,d)\to 0,$ as $t\to \infty.$

By choosing $t$ such that $$e^{-\frac12 \zeta(D_mt, \delta)}=\frac{1}{2}\frac{\sqrt{m(A_1)m(A_2)}}{m(V)},$$ we have
$$\lambda_2\leq \frac{1}{t}\log \frac{2m(V)}{\sqrt{m(A_1)m(A_2)}}.$$ By the homogeneity of $\zeta,$ $\zeta(t,d)=d\zeta(\frac{t}{d},1),$ and the definition of $h,$  we know that
given $d,a>0,$ the solution of $\zeta(t,d)=a$ is $t=dh(\frac{a}{d}).$ This implies that $$t=\frac{\delta}{D_m} h\left(\frac{2}{\delta}\log \frac{2m(V)}{\sqrt{m(A_1)m(A_2)}}\right).$$

Hence $$\lambda_2\leq \frac{D_m}{\delta}\frac{\log \frac{2m(V)}{\sqrt{m(A_1)m(A_2)}}}{h \left(\frac{2}{\delta}\log \frac{2m(V)}{\sqrt{m(A_1)m(A_2)}}\right)}.$$

{\bf Case 2.} $k>2.$
Using generalized Fourier expansion w.r.t. the orthonormal basis $\{\phi_i\},$ one has
$$\mathds{1}_{A_j}=\sum_{i=1}^Na_j^i \phi_i,\ \ j=1,\cdots,k.$$ where $a_j^i=(\mathds{1}_{A_j},\phi_i)_{\ell^2(V,m)}.$  By the same argument in Case 1, for any $1\leq j\neq l\leq k$ we have
\begin{eqnarray*}\sum_{x\in A_j}\sum_{y\in A_l}p_t(x,y)m(x)m(y)&=&a^1_ja_l^1+\sum_{i=2}^{k-1}e^{-\lambda_i t}a^i_ja^i_l+\sum_{i=k}^{N}e^{-\lambda_i t}a^i_ja^i_l\\
&\geq&\frac{m(A_j)m(A_l)}{m(V)}+\sum_{i=2}^{k-1}e^{-\lambda_i t}a^i_ja^i_l-e^{-\lambda_kt}\sqrt{m(A_j)m(A_l)}.
\end{eqnarray*}
Combining this with \eqref{e: Davies gamma=1} in the DGG Lemma, we obtain the following by direct calculation
\begin{equation}\label{e:for high eigenvalues}
e^{-\lambda_k t}\geq \frac{\sqrt{m(A_j)m(A_l)}}{m(V)}+\frac{1}{\sqrt{m(A_j) m(A_l)}}\sum_{i=2}^{k-1}e^{-\lambda_i t}a^i_ja^i_l-e^{-\frac12 \zeta(D_m t, d(A_j,A_l))}.
\end{equation}
We choose $t_0>0$ such that
\begin{equation}\label{e:high eigenvalues1}e^{-\frac12\zeta(D_mt_0,\delta)}=\frac12\min_{j\neq l} \frac{\sqrt{m(A_j)m(A_l)}}{m(V)},\end{equation} where $\delta=\max_{j\neq l}d(A_j,A_l).$ Using the inverse function, $h(t),$ of $\zeta(t,1),$ one finds that
\begin{equation}\label{e:high eigenvalues2}t_0=\frac{\delta}{D_m}\min_{j\neq l} h\left(\frac{2}{\delta}\log \frac{2m(V)}{\sqrt{m(A_j)m(A_l)}}\right).\end{equation} The reason for this choice of $t_0$ will be apparent soon.

We claim that there exists a pair $\{j_0,l_0\},$ $1\leq j_0\neq l_0\leq N,$ such that for $A_{j_0}$ and $A_{l_0}$ the second term on the right hand side of the equation \eqref{e:for high eigenvalues} is nonnegative. For this purpose, we consider an auxiliary vector space, $\R^{k-2},$ endowed with the inner product
$$\langle x,y \rangle_{t_0}=\sum_{i=1}^{k-2}e^{-\lambda_{i+1}t_0}x_iy_i, \ \ x,y\in \R^{k-2}.$$
We have $k$ vectors, $\left\{X_j=(a_j^2,a_j^3,\cdots,a_j^{k-1})\right\}_{j=1}^k$, in $\R^{k-2}.$ By a standard theorem in linear algebra, see \cite[Lemma~2]{Chung-Grigoryan-Yau96}, there exists a pair $1\leq j_0\neq l_0\leq k$ such that $\langle X_{j_0},X_{l_0}\rangle_{t_0}\geq 0,$ that is
$$\sum_{i=2}^{k-1}e^{-\lambda_i t}a^i_{j_0}a^i_{l_0}\geq 0.$$ This proves the claim.

For the pair $j_0$ and $l_0,$ it follows from \eqref{e:for high eigenvalues} with $t=t_0$ that
\begin{eqnarray*}e^{-\lambda_k t_0}&\geq& \frac{\sqrt{m(A_{j_0})m(A_{l_0})}}{m(V)}-e^{-\frac12 \zeta(D_m t_0, d(A_{j_0},A_{l_0}))}\\
&\geq& \min_{j\neq l}\frac{\sqrt{m(A_{j})m(A_{l})}}{m(V)}-e^{-\frac12 \zeta(D_m t_0, \delta)}\\
&=&\frac12 \min_{j\neq l}\frac{\sqrt{m(A_{j})m(A_{l})}}{m(V)},\end{eqnarray*}
where we use the monotonicity of $\zeta$ in $d$ in the second inequality and the property \eqref{e:high eigenvalues1} of $t_0$ for our choice in the last equality.
Combining this with \eqref{e:high eigenvalues2}, we prove the theorem.
\end{proof}
Finally, we will give an example to show the sharpness of the estimate of the order $1/\delta^2$ in Theorem \ref{thm:eigenvalue estimates}.
\begin{example}\label{Example sharp}\begin{enumerate}[1.]\item ($k=2$). For any $n\in \N,$ let $P_{4n+1}$ be a path graph identified with the induced subgraph $[-2n,2n]\cap \Z$ of $\Z.$ We choose $A_1=[-2n,-n]\cap \Z,$ $A_2=[n,2n]\cap \Z.$ Then $$\frac{2}{\delta}\log\left(\frac{2m(V)}{\sqrt{m(A_1)m(A_2)}}\right)\sim \frac{2\log8}{n}\ll1.$$ By our estimate \eqref{Eigenvalue estimate}, $\lambda_2\leq \frac{C}{n^2}\sim \frac{C}{\mathrm{diam} ^2}$ which is optimal for large $n$ since $\lambda_2 = 1 - \cos(
\frac{\pi}{n-1})$.
\item $(k>2).$ For $k$ copies of the path graph $[0,2n]\cap \Z,$ $\{G_l\}_{l=1}^k$, we glue the origins of $G_l$ together to get a star graph $G.$ By setting $A_l=[n,2n]\cap G_l$ for $1\leq l\leq k,$ our estimate \eqref{Eigenvalue estimate} implies that $\lambda_k\leq \frac{C}{n^2}$ which is known to be optimal.
\end{enumerate}
\end{example}
We briefly discuss some consequences of Theorem \ref{thm:eigenvalue estimates}.
\begin{coro}The diameter $D$ of a graph satisfies
$$D\leq 2 \left(\frac{D_m}{\sigma\arcsinh(\sigma^{-1})\lambda_2}\right)^{1/2}\log\left(\frac{2m(V)}{m_\mathrm{min}}\right),$$ where we can choose
$$\sigma = \left[\sinh\left(4\log\left(\frac{2m(V)}{m_\mathrm{min}}\right)\right)\right]^{-1}.$$
\end{coro}
\begin{proof}
The proof follows immediately from Theorem \ref{thm:eigenvalue estimates} and Remark \ref{rem: eigenvalue estimates} by choosing $k=2$, $A_1 =\{x\}$ and $A_2=\{y\}$ where $d(x,y)=D$.
\end{proof}
Using Theorem \ref{thm:eigenvalue estimates} we can easily derive isoperimetric inequalities that improve and generalize earlier results in \cite{Chung-Grigoryan-Yau96, Alon85, Tanner84}. For a subset $U\subset V$, the $r$-neighborhood of $U$ is defined by
$$N_r(U) =\{x\in V: d(x,U)\leq r\}.$$
\begin{coro}We have the following lower bound for the size of the $r$-neighborhood of a subset $U\subset V,$ $r\geq 1$
$$m(N_r(U))\geq m(V)\left(1 - \frac{4m(V)}{m(U)}\exp\left(-(r+1)\sqrt{\frac{\lambda_2 \sigma\arcsinh(\sigma^{-1})}{D_m}}\right)\right),$$
 where we can choose $$\sigma= \left(\sinh\left(2\log\frac{2m(V)}{\sqrt{m(U)m_\mathrm{min}}}\right)\right)^{-1}.$$\end{coro}
\begin{proof}
The proof follows immediately from Theorem \ref{thm:eigenvalue estimates} and Remark \ref{rem: eigenvalue estimates} by choosing $k=2$, $A_1 =U$ and $A_2=V\setminus N_r(U)$.
\end{proof}
\bibliography{Literature}
\bibliographystyle{alpha}

\end{document}